\documentclass{amsart}

\newcommand{\id}{\mathrm{id}}
\newcommand{\U}{\mathcal U}
\newcommand{\cf}{\mathrm{cf}}
\newcommand{\w}{\omega}
\newcommand{\Ra}{\Rightarrow}
\newcommand{\F}{\mathcal F}
\newcommand{\supp}{\mathrm{supp}}
\newcommand{\pr}{\mathrm{pr}}
\newcommand{\e}{\varepsilon}

\newcommand{\IN}{\mathbb N}

\newcommand{\IZ}{\mathbb Z}

\newtheorem{theorem}{Theorem}[section]

\newtheorem{problem}[theorem]{Problem}
\newtheorem{claim}[theorem]{Claim}
\newtheorem{proposition}[theorem]{Proposition}
\theoremstyle{definition}
\newtheorem{example}[theorem]{Example}
\newtheorem{remark}[theorem]{Remark}
\newtheorem{definition}[theorem]{Definition}

\title{Topologizations of a set endowed with an action of a monoid}
\author{Taras Banakh, Igor Protasov, Olga Sipacheva}
\address{T. Banakh: Ivan Franko National University of Lviv (Ukraine) and
Jan Kochanowski University in Kielce (Poland)}
\email{t.o.banakh@gmail.com}
\address{I. Protasov:
Taras Shevchenko National University, Kyiv (Ukraine)}
\email{i.v.protasov@gmail.com}
\address{O. Sipacheva: Lomonosov Moscow State University, Moscow (Russia)}
\email{o-sipa@yandex.ru}
\date{}
\dedicatory{Dedicated to Mitrofan Choban and Stoyan Nedev on the occasion of their 70th
birthdays}
\keywords{Zariski topology, $G$-act, non-discrete topologization}
\subjclass{54H15}

\begin{document}

\begin{abstract}
Given a set $X$ and a family $G$ of self-maps of $X$, we study the problem of the
existence of a non-discrete Hausdorff topology on $X$ with respect to which all functions $f\in G$
are continuous. A topology on $X$ with this property is called a {\em $G$-topology}. The answer is
given in terms of the Zariski $G$-topology $\zeta_G$ on $X$, that is, the topology generated
by the subbase consisting of the sets $\{x\in X:f(x)\ne g(x)\}$ and $\{x\in X:f(x)\ne c\}$, where
$f,g\in G$ and $c\in X$. We prove that, for a countable monoid $G\subset X^X$, $X$ admits a
non-discrete Hausdorff $G$-topology if and only if the Zariski $G$-topology $\zeta_G$ is
non-discrete; moreover, in this case, $X$ admits $2^{\mathfrak c}$ hereditarily normal $G$-topologies.
\end{abstract}

\maketitle

\section{Principal Problems}
In this paper we consider the following general problem.

\begin{problem}\label{pr1} Given a set $X$ and a family $G$ of self-maps of $X$, determine whether
$X$ admits a non-discrete Hausdorff (or normal) topology with respect to which all functions $g\in
G$ are continuous. \end{problem}

Since the composition of continuous functions is continuous, we can assume without loss of
generality that the family $G$ is a subsemigroup in the semigroup $X^X$ of all functions $X\to X$
endowed with the operation of composition. Also, we can assume that $G$ contains the identity
function $\id_X$ of $X$ and, hence, is a submonoid of $X^X$. Thus, it is natural to consider
Problem~\ref{pr1} in the context of actions of monoids.

Let $G$ be a monoid with two-sided unit $1_G$. A {\em left
unitary action},  of a monoid $G$ on a set $X$ is a function $\alpha:G\times X\to X$,
$\alpha:(g,x)\mapsto g(x)$, with the following two properties:
\begin{itemize}
\item $f(g(x))=(fg)(x)$ for all $f,g\in G$ and $x\in X$  and
\item $1_G(x)=x$ for all $x\in X$.
\end{itemize}
A set $X$ endowed with an action of a monoid $G$  is called a (left unitary) {\em $G$-act},
or an {\em act over $G$} (this is a standard term of semigroup theory; see, e.g., \cite{KKM}). A
topology $\tau$ on a $G$-act $X$ is called a {\em $G$-topology} if, for every $g\in G$, the shift
$g:X\to X$, $g:x\mapsto g(x)$, is continuous.
 A $G$-act $X$ is said to be ({\em normally}) {\em $G$-topologizable}
if $X$ admits a non-discrete (normal) Hausdorff $G$-topology. A topology $\tau$ on a set $X$ is
called {\em normal} if the topological space $(X,\tau)$ is normal, that is, $X$ is a $T_1$-space
such that any two disjoint closed subsets in $X$ have disjoint open neighborhoods.
A topology $\tau$ on a set $X$ is {\em hereditarily normal} if each subspace
of the topological space $(X,\tau)$ is normal.

In this terminology, Problem~\ref{pr1} can be rewritten as follows.

\begin{problem}\label{pr2} Find  necessary and sufficient conditions for the (normal)
$G$-topologizability of a given $G$-act $X$.
\end{problem}

For acts over a group $G$, this problem was considered in \cite{BPS}.

Problem~\ref{pr2} is motivated by Markov's celebrated problem on the existence of a
non-discrete Hausdorff group topology on an arbitrary infinite group~\cite{Markov1945}, which is
closely related to the problem of whether any unconditionally closed subset of a group $G$ (i.e., a
set closed in any Hausdorff group topology on $G$) is algebraic (i.e., is an intersection of finite
unions of solution sets of equations in $G$), which was also posed by Markov in~\cite{Markov1945}.
By the definition of a group topology, any algebraic set must be unconditionally closed; on the
other hand, a group admits a non-discrete Hausdorff group topology if and only if the complement to
the identity (or any other) element in this group is not unconditionally closed. In \cite{Mar} (see
also \cite{Markov1944}), Markov proved the coincidence of unconditionally closed sets with
algebraic sets for countable groups; a general (negative) answer was given only in 1979 by
Hesse~\cite{Hesse}, who constructed an example of an uncountable  non-topologizable group in which
the complement to the identity element is not algebraic. The first countable
non-topologizable group was given in 1980 by Ol'shanskii~\cite{Ol80}; other examples
were constructed in~\cite{KT} and~\cite{KOO}. In contrast to these negative results, Zelenyuk \cite{Zel07} proved that each infinite group $G$ admits a non-discrete regular topology with continuous shifts and continuous inversion.


The family of algebraic subsets of a group $G$ coincides with the family
of closed subsets of a $T_1$ topology $\zeta_G$ on $G$, called the \emph{Zariski topology}.
Markov did not explicitly introduce this topology, although he
implicitly considered it in the form of algebraic closures of sets~\cite{Mar, Markov1945}.
The Zariski topology was explicitly introduced only in 1977 by Bryant \cite{Br} under the name
verbal topology; the term was introduced by Dikranjan and Shakhmatov~\cite{DikShakh}.
This topology was studied in, e.g., \cite{Br, Sipa, DS2,
DT1, DT}.

The family of unconditionally closed subsets of $G$ coincides with the family of closed subsets in
another $T_1$ topology on $G$, namely, the infimum of all Hausdorff group topologies on $G$.
This topology was explicitly introduced in \cite{DS1} under the name {\em Markov
topology}.

Thus, Markov's theorem on the coincidence of algebraic and unconditionally closed sets in
countable groups asserts that, in such groups, the Zariski topology coincides with the Markov
topology, or, in other words, that a countable group admits a non-discrete Hausdorff topology
if and only if the Zariski topology of this group is non-discrete. In this paper, we obtain a
similar result for $G$-acts (see Theorem~\ref{main}); namely, we prove that an act over a countable
monoid $G$ is $G$-topologizable if and only if the Zariski $G$-topology on $X$, which is defined
and studied in Section~\ref{s:zar}, is non-discrete. We also show that, in this case, $X$ admits
a non-discrete metrizable $G$-topology and $2^{\mathfrak c}$ non-discrete hereditarily normal
$G$-topologies. This fact can be compared with a result of Dikranjan and Protasov \cite{DP} saying that each topologizable countable group $G$ admits $2^{\mathfrak c}$ (pairwise transversal) topologies.

The paper is organized as follows. In Section~\ref{s:zar}, we introduce the Zariski topology on an
$S$-act over a monoid $S$ and consider particular cases of $S$-acts $G$, where $G$ is a semigroup
or a group and $S$ is a submonoid of $G^G$. The technical Section~\ref{s:spec} describes
properties of maximal $G$-topologies on $G$-acts in which the so-called discriminate filters (defined
in the same section) converge. Section~\ref{s:existence} contains conditions for the
existence of discriminate filters on $G$-acts, given in terms of the Zariski topology $\zeta_G$. The last section contains the proofs of the main results on the topologizability of
$G$-acts.

\section{The Zariski $G$-topology on a $G$-act}\label{s:zar}

In this section we define the Zariski $G$-topology on a $G$-act $X$ and study this topology for
particular examples of $G$-acts.

\begin{definition} For an act $X$ over a monoid $G$, the {\em Zariski $G$-topology} $\zeta_G$ on
$X$ is the topology generated by the subbase $\tilde\zeta_G$ consisting of the sets $\{x\in
X:f(x)\ne g(x)\}$ and $\{x\in X:f(x)\ne c\}$, where $f,g\in G$ and $c\in X$.
\end{definition}

The following easy fact follows directly from the definition.

\begin{proposition}
\label{Zarisweak}
For any $G$-act $X$, the Zariski $G$-topology $\zeta_G$ satisfies the
separation axiom $T_1$ and is contained in any Hausdorff $G$-topology on $X$.
\end{proposition}

An immediate corollary of this observation is that $\zeta_G$ is always non-discrete for a
$G$-topologizable $G$-act. In the case of countable $G$-acts, the non-discreteness of $\zeta_G$
turns out to be sufficient for $G$-topologizability (see Theorem~\ref{main}).

The subbase
$\tilde\zeta_G$ of the Zariski $G$-topology $\zeta_G$
has a natural cardinal invariant $\psi(x,\tilde\zeta_G)$, called the pseudocharacter of
$\tilde\zeta_G$ at a point $x\in X$. In fact,  $\psi(x,\F)$ can be defined for
any family $\F$ of subsets of $X$.
Given a point $x\in X$, we set
$$
\F_x=\{X\}\cup\{F\in\F:x\in F\}
$$
and define the {\em pseudocharacter} of $\F$ at $x$ as
$$
\psi(x,\F)=\min\{|\U|:\U\subset\F_x\mbox{ \ and ${\bigcap}\,\U=\bigcap\,\F_x$}\}.
$$
If $\tau$ is the topology on $X$ generated by a subbase $\F$, then $\tau_x$ is the family of
all open neighborhoods of $x$, and $\psi(x,\tau)$ is the usual pseudocharacter
of the topological space $(X,\tau)$  at the point $x$. It is easy to see that
$\psi(x,\tau)=\psi(x,\F)$ for any non-isolated point $x$ in $(X,\tau)$. If $x$ is isolated in
$(X,\tau)$, then $\psi(x,\tau)=1$, while $1\le \psi(x,\F)<\aleph_0$, so the pseudocharacter
$\psi(x,\tilde\zeta_G)$ carries more information than $\psi(x,\zeta_G)$ for an isolated
point $x$ in $(X,\zeta_G)$. If $x$ is non-isolated, then $\psi(x,\tilde\zeta_G)=\psi(x,\zeta_G)$.

In the algebraic language, the pseudocharacter $\psi(x,\tilde \zeta_G)$ equals the least number
of inequalities of the form
$$
\mbox{$f(x)\ne g(x)$ \  or \ $f(x)\ne c$, \ where \ $f,g\in G$ and $c\in
X$,}
$$
in a system of inequalities whose unique solution is $x$. Note that the pseudocharacter of
the Zariski topology is tightly related to the notion of being ungebunden, which
was introduced for general algebraic systems by Podewski~\cite{Podewski} and corrected and
generalized by Hesse~\cite{Hesse}. Namely, an algebraic system  is said to be
\emph{$\kappa$-ungebunden} if the pseudocharacter of the naturally defined Zariski topology on
this system is at least $\kappa$ at each point. Podewski proved that any $|G|$-ungebunden
algebraic system $G$ admits $2^{2^{|G|}}$ Hausdorff topologies consistent with the
algebraic structure, and Hesse investigated $\kappa$-ungebunden systems and constructed examples
showing that Podewski's sufficient topologizability condition cannot be weakened in the sense
that, for any cardinals $\kappa$ and $\lambda$ such that $\cf(\lambda)>\kappa=\cf(\kappa)\ge
\aleph_0$, there exists a group $G$ with $|G|=\lambda$ which is $\kappa$-ungebunden (that is,
$\psi(\zeta_G) = \kappa$) and admits no Hausdorff group (and even semigroup) topologies.

Now, let us consider the Zariski $G$-topology for some particular $G$-acts.

\begin{example} Let $X$ be an infinite set endowed with the natural action of the group $G$ of all
bijective functions $f:X\to X$ that have finite support
$$
\supp(f)=\{x\in X:f(x)\ne x\}.
$$
It is easy to see that $\psi(x,\zeta_G)=1$ and $\psi(x,\tilde\zeta_G)=2$ for any point $x\in X$.
Consequently, the Zariski $G$-topology $\zeta_G$ on $X$ is discrete, and the $G$-act $X$ is not
$G$-topologizable.
\end{example}

Each semigroup $G$ can be considered as an $S$-act for many natural actions of various submonoids
$S$ of the monoid $G^G$. We define five such natural submonoids of $G^G$:
\begin{itemize}
\item $G_l$ is the smallest submonoid of $G^G$ containing all left shifts
$l_a:x\mapsto ax$ of $G$ for $a\in G$;
\item $G_r$ is the smallest submonoid of $G^G$ containing all right shifts
$r_a:x\mapsto xa$ of $G$ for $a\in G$;
\item $G_s$ is the smallest submonoid of $G^G$  containing
all two-sided shifts $s_{a,b}:x\mapsto axb$ of $G$ for $a,b\in G$;
\item $G_m$, where $m\in\IN$, is
the smallest submonoid of $G^G$ containing $G_s$ and the $m$th power map $x\mapsto x^m$;
\item $G^+_p$ is the smallest submonoid of $G^G$ containing $G_s$ and such that
the product $f\cdot g:x\mapsto f(x)\cdot g(x)$ of any two functions $f,g\in G_p^+$
belongs to $G_p^+$.
\end{itemize}
Clearly,
$$
G_l\cup G_r\subset G_s\subset G_m\subset G^+_p,
$$
and hence
$$
\zeta_{G_l}\cup\zeta_{G_r}\subset\zeta_{G_s}\subset\zeta_{G_m}\subset\zeta_{G^+_p}.
$$

For a group $G$, we augment the above list of submonoids of $G^G$ by
some other submonoids:
\begin{itemize}
\item $G_q$ is the subgroup of $G^G$ containing all
bijections of the form $f:x\mapsto ax^\e b$, where $a,b\in G$ and $\e\in\{1,-1\}$;
\item $G_m$, where $m\in\IZ$, is
the smallest submonoid of $G^G$ containing $G_s$ and the $m$th power map $x\mapsto x^m$;
\item $G_{[s]}$ is the smallest submonoid of $G^G$ containing the subgroup $G_s$
and the maps $\gamma_a:G\to G$,\break $\gamma_a:x\mapsto xax^{-1}$, for all $a\in G$;
\item $G_{[q]}$ is the smallest submonoid of $G^G$ containing the subgroup $G_q$
and the maps $\gamma_a:G\to G$,\break $\gamma_a:x\mapsto xax^{-1}$, for all $a\in G$;
\item $G_{p}$ is the smallest submonoid containing the subgroup $G_q$ and such that
the  product $f\cdot g:x\mapsto f(x)\cdot g(x)$ of any two functions $f,g\in G_p$ belongs
 to $G_{p}$.
\end{itemize}
We refer to functions from the family $G_{p}$ as {\em polynomials} on the group $G$.

$G_l$- and $G_r$-topologies on a group $G$ are known as left- and right-invariant topologies;
a group $G$ endowed with a $G_l$-topology ($G_r$-topology, $G_s$-topology, $G_q$-topology) is
said to be {\em left-topological} (respectively, {\em right-topological}, {\em semi-topological},
{\em quasi-topological\/}).
Following \cite{BGP}, we refer to a group $G$ endowed with a
$G_{[s]}$-topology ($G_{[q]}$-topology) as a [{\em semi\/}]-{\em topological} (respectively,
[{\em quasi\/}]-{\em topological}) group.

Now, consider the structure of the Zariski $S$-topologies on a group $G$ for
$S\in\{G_l, G_r, G_s, G_{[s]},G_q,G_{[q]},G^+_p,\allowbreak G_{p}\} \cup \{G_m:m\in\IZ\}$. It should be mentioned that $G_0=G_1=G_s$ and $G_{-1}=G_q$.

By the {\em cofinite topology} on a set $X$ we understand the topology
$$
\tau_1=\{\emptyset\}\cup\{X\setminus F: \mbox{$F$ is a  finite subset of $X$}\}.
$$

The following assertion follows easily from the definitions.

\begin{remark} For any group $G$, the Zariski topologies $\zeta_{G_l}$ and $\zeta_{G_r}$ on $G$
coincide with the cofinite topology on $G$. If $G$ is infinite, then the topologies $\zeta_{G_l}$
and $\zeta_{G_r}$ are not Hausdorff.
\end{remark}

\begin{remark}\label{r2.5}
For any infinite group $G$, the Zariski topologies $\zeta_{G_s}$ and
$\zeta_{G_q}$ are not discrete. This follows from a deep result of Zelenyuk \cite{Zel,
Zel07}, who proved that each infinite group $G$ admits a non-discrete Hausdorff topology with
continuous two-sided shifts and continuous inversion.
\end{remark}

\begin{remark} For a group $G$ endowed with the natural action of the monoid $G_{p}$ of all
polynomial functions on $G$, the Zariski $G_{p}$-topology $\zeta_{G_{p}}$ coincides with the usual
Zariski topology $\mathfrak Z_G$ on the group $G$.
In particular, for countable non-topologizable groups,
the Zariski $G_{p}$-topology $\zeta_{G_{p}}$ is
discrete, while for Hesse's example of an uncountable non-topologizable
group $G$ mentioned above (in which
the complement to the identity element is not algebraic),
the Zariski topology $\zeta_{G_{p}}$ is non-discrete.
\end{remark}

Thus, for infinite groups $G$, the Zariski $G_q$-topology $\zeta_{G_q}$ is always non-discrete,
while the topology $\zeta_{G_p}$ may be discrete (e.g., for countable non-topologizable
groups).

If a group $G$ is Abelian, then the Zariski topology $\zeta_{G_s}$ on $G$ coincides with the
topologies $\zeta_{G_l}$ and $\zeta_{G_r}$ and is therefore cofinite. However, for non-Abelian
groups $G$, the topology $\zeta_{G_s}$ may have rather unexpected properties.

\begin{example}[Dikranjan--Toller]\label{ex2.7}
Let $H$ be a finite discrete topological group with
trivial center (for example, let $H=\Sigma_3$ be the group of bijections of a 3-element set). For
any cardinal $\kappa$, the Zariski topologies $\zeta_{G_s}$ and $\zeta_{G_p}$ on the group
$G=H^\kappa$ coincide with the Tychonoff product topology $\tau$ on $G=H^\kappa$; hence these
topologies are compact and  Hausdorff and have pseudocharacter
$\psi(x,\tau)=\kappa<2^\kappa=|G_s|=|G_p|=|G|$ at each point $x\in G$.
\end{example}

\begin{proof} Observe that the Tychonoff product topology $\tau$ on $G=H^\kappa$ turns the group
$G$ into a compact topological group. Thus, each polynomial map on $G$ is continuous, and each set
$U\in\tilde\zeta_{G_p}$ is open in $G$ with this topology. Consequently,
$\zeta_{G_s}\subset\zeta_{G_p}\subset\tau$. The Tychonoff product  topology $\tau$ is generated by
the subbase consisting of the sets
$$
U_{\alpha,h}=\{x\in G:\pr_\alpha(x)=h\},
$$
where $\alpha\in\kappa$, $h\in H$, and
$\pr_\alpha:H^\kappa\to H$ denotes the $\alpha$th coordinate projection. To prove that
$\zeta_{G_s}=\zeta_{G_p}=\tau$, it suffices to check that each set $U_{\alpha,h}$ belongs to the
topology $\zeta_{G_s}$.

Consider the embedding $i_\alpha:H\to H^\kappa$ which takes each $x\in H$
to the point $i_\alpha(x)\in H^\kappa$ such that $\pr_\alpha\circ i_\alpha(x)=x$ and
$\pr_\beta\circ i_\alpha(x)=1_H$ for all $\beta\ne \alpha$.

Given a point $h\in H$, consider the finite set $A_h=\{(a,b)\in H\times H:ah\ne hb\}$ and observe
that $\{h\}=\bigcap_{(a,b)\in A_h}\{x\in H:x^{-1}ax\ne b\}$. Indeed, it follows from the
triviality of the center of $H$ that, for any $x\in H\setminus\{h\}$, there is an element $a\in H$
such that $(xh^{-1})a\ne a(xh^{-1})$ and, therefore,  $h^{-1}ah\ne x^{-1}ax$. For $b=x^{-1}ax$,
we have $h^{-1}ah\ne b$, whence $(a,b)\in A_h$.

For each pair $(a,b)\in A_h$, consider the left and right shifts $l_{a}:x\mapsto i_\alpha(a)\cdot
x$ and $r_b:x\mapsto x\cdot i_\alpha(b)$ of the group $G=H^\kappa$. These shifts generate the
subbase set
$$
U_{a,b}=\{x\in G: i_\alpha(a)\cdot x\ne x\cdot i_\alpha(b)\}=\{x\in G: l_a(x)\ne
r_b(x)\}\in\tilde\zeta_{G_s}.
$$
It remains to note that
$$
\pr_{\alpha}^{-1}(h)=\bigcap_{(a,b)\in A_h}U_{a,b}\in\zeta_{G_s};
$$
therefore, $\tau=\zeta_{G_p}=\zeta_{G_s}$. The equality
$\zeta_{G_s}=\zeta_{G_p}$ implies
$\zeta_{G_s}=\zeta_{G_{[s]}}=\zeta_{G_q}=\zeta_{G_{[q]}}=\zeta_{G_p^+}=\zeta_{G_p}$.
\end{proof}

Next, we consider the Zariski topologies on permutation groups.

\begin{remark}
Given a  set $X$ of cardinality $|X|\ge 3$, let $S(X)$ be the group of all bijective
transformations of $X$, and let $S_\w(X)$ be the normal subgroup of $S(X)$ consisting of all
bijective transformations $f:X\to X$ with finite support $\supp(f)=\{x\in X:f(x)\ne x\}$. According
to \cite{BGP}, for any group $G$ with $S_\w(X)\subset G\subset S(X)$ and any monoid
$S\in\{G_{[s]},G_{[q]},G_{p}\}$, the Zariski $S$-topology $\zeta_S$ on $G$ coincides with the
topology of pointwise convergence $\mathcal T_p$ and therefore is completely regular. If $X$ is
infinite, then the Zariski topologies $\zeta_{G_s}$ and $\zeta_{G_{[s]}}$ are distinct, as shown in
(the proof of) Lemma 2.9 in~\cite{BGP}.
\end{remark}

\begin{remark}
According to \cite[6.3]{BGP}, for any set $X$, the topology $\zeta_{G_s}$ on the permutation
group $G=S_\w(X)$ is $\sigma$-discrete (i.e., $G$ can be represented as a countable union
of discrete subspaces of $(G,\zeta_G)$\,). Consequently, for any submonoid $S\subset G^G$
containing $G_s$, the $S$-topology $\zeta_S$ on $G=S_\w(X)$ is $\sigma$-discrete.
Permutation groups $S_\w(X)$ belong to the class of perfectly supportable semigroups, introduced in
\cite{BG}. According to \cite[3.5]{BG}, for each perfectly supportable semigroup $G$, the
topological space $(G,\zeta_{G_s})$ is $\sigma$-discrete.
\end{remark}

By Zelenyuk's result \cite{Zel07}, mentioned in Remark~\ref{r2.5}, for every infinite group $G$ the Zariski topologies $\zeta_{G_0}=\zeta_{G_1}=\zeta_{G_s}$ and $\zeta_{G_{-1}}=\zeta_{G_q}$ are not discrete. On the other hand, we have:

\begin{proposition} For any non-zero number
$m\in\IZ\setminus\{-2^n,2^n:n\in\w\}$ there is a countable infinite group $G$
with discrete Zariski topology $\zeta_{G_m}$.
\end{proposition}

\begin{proof} By assumption, the number $m$ has a prime divisor $p\ge 3$.
Let us write $m$ as $m=l\cdot p^\alpha$, where $\alpha$ is a positive integer
and $l$ is an integer not divisible by $p$, and choose
a positive integer $\beta$ such that $p^{\alpha\beta}\ge 665$. According to \cite{Ol80} and
\cite{Ol89}, there is a countable infinite group $G$ containing a cyclic
subgroup $C$ of order $|C|=p^{\alpha\beta}$ such that
$$
G\setminus \{1_G\}=(C\setminus\{1_G\})
\cup \bigcup_{c\in C\setminus\{1_G\}}\{x\in G:x^{p^{\alpha\beta}}=c\}.
$$
Since $l^\beta$ is coprime to $p^{\alpha\beta}$, it follows that, given any element $c\in C$,
we have $c\ne 1_G$ if and only if $c^{l^\beta}\ne 1_G$. This implies
$$
G\setminus \{1_G\}=(C\setminus\{1_G\})
\cup\bigcup_{c\in C\setminus\{1_G\}}\{x\in G:(x^{p^{\alpha\beta}})^{l^\beta}=c\}.
$$
The continuity of the map $x\mapsto x^m$ (with respect to the topology $\zeta_{G_m}$)
implies the continuity of the map $f:G\to G$, $f:x\mapsto x^{m^\beta}=x^{p^{\alpha\beta} l^\beta}$.
Therefore, the set $G\setminus\{1_G\}=(C\setminus\{1_G\})\cup f^{-1}(C\setminus\{1_G\})$
is closed in the topology $\zeta_{G_m}$, and its complement $\{1_G\}$ is open,
which implies the discreteness of the topology $\zeta_{G_m}$.
\end{proof}

Thus, for every infinite group $G$, the Zariski topologies $\zeta_{G_m}$, $m\in\{-1,0,1\}$,
are not discrete, while for any odd $m\notin\{-1,1\}$, there is an infinite group $G$
with discrete Zariski topology $\zeta_{G_m}$.

\begin{problem}  Does there exist an infinite group $G$
with discrete Zariski topology $\zeta_{G_2}$? Equivalently, is there an infinite group $G$ admitting no non-discrete shift-invariant Hausdorff topology with continuous map $x\mapsto x^2$?
\end{problem}

By the result of Zelenyuk mentioned in Remark~\ref{r2.5}, each infinite group $G$ is
$G_s$-topologizable.  The following simple
example shows that this result cannot be generalized to semigroups.

\begin{example} For the monoid $G=(\IN,\min)$, the Zariski topology $\zeta_{G_s}$ (which coincides
with $\zeta_{G_l}$ and $\zeta_{G_r}$) is discrete. Indeed, for each $n\in\IN$, the singleton
$\{n\}$ belongs to the topology $\zeta_{G_s}$, because
$$
\{n\}=\{x\in \IN:\min\{x,n\}\ne
n\}\cap\bigcap_{k<n}\{x\in \IN:x\ne k\}.
$$
This implies that the monoid $G=(\IN,\min)$ is not $G_s$-topologizable.
\end{example}

\begin{example} For the monoid $G=(\IN,\max)$, the Zariski topology $\zeta_{G_s}$ (which coincides
with $\zeta_{G_l}$ and $\zeta_{G_r}$) is discrete. Indeed, for each $n\in\IN$, the singleton
$\{n\}$ belongs to the topology $\zeta_{G_s}$, because
$$
\{n\}=\{x\in \IN:\max\{x,n\}\ne
x\}\cap\bigcap_{k<n}\{x\in \IN:x\ne k\}.
$$
This implies that the monoid $G=(\IN,\max)$ is not $G_s$-topologizable.
\end{example}

\section{$G$-topologies on $G$-acts generated by discriminate filters}\label{s:spec}

In this section we describe and study $G$-topologies on $G$-acts generated by filters of a special
form.

A {\em filter} on a set $X$ is a family $\varphi$ of subsets of $X$ such that
\begin{itemize}
\item $\emptyset\notin\varphi$;
\item $A\cap B\in\varphi$ for any sets $A,B\in\varphi$;
\item $A\cup B\in\varphi$ for any sets $A\in\varphi$ and $B\subset X$.
\end{itemize}

By the {\em pseudocharacter $\psi(\varphi)$ of a filter} $\varphi$ we understand the smallest
cardinality $|\F|$ of a subfamily $\F\subset\varphi$ such that $\bigcap\F=\bigcap\varphi$. The {\em
character} $\chi(\varphi)$ of a filter $\varphi$ equals the smallest cardinality of a subfamily
$\F\subset\varphi$ such that each set $\Phi\in\varphi$ contains some set $F\in\F$. Note that
the {\em character} $\chi(x,\tau)$ of a topological space $(X,\tau)$  at a point $x$ can be defined
as the character $\chi(\tau_x)$ of the neighborhood filter $\tau_x=\{U\in\tau:x\in U\}$.

Given a filter $\varphi$ on $X$, consider the family
$$
\varphi^+=\{E\subset X:\forall F\in\varphi\;\;F\cap E\ne\emptyset\}
$$
equal to the union of all
filters on $X$ that contain $\varphi$. It is easy to check that, for each $A\subset  X$ with
$A\notin\varphi$, we have $X\setminus A\in\varphi^+$.

\smallskip

A filter $\varphi$ on a topological space $X$ is said to {\em converge to a point} $x_0$ if
each neighborhood $U\subset X$ of $x_0$ belongs to $\varphi$.

Now, suppose that $G$ is a monoid, $X$ is a $G$-act, and $\varphi$ is a filter on $X$ such
that $\bigcap\varphi=\{x_0\}$ for some point $x_0$. Then we can consider the largest $G$-topology
$\tau_\varphi$ on $X$ in which the filter $\varphi$ converges to $x_0$. This topology admits
the following simple description.

\begin{proposition}
The topology $\tau_\varphi$ consists of all sets $U\subset X$ such that, for any
$g\in G$ with $x_0\in g^{-1}(U)$, the preimage $g^{-1}(U)$ belongs to the filter
$\varphi$.
\end{proposition}

Our strategy is to find some class of filters $\varphi$ on $X$ which generate $G$-topologies
$\tau_\varphi$ on $X$.

\begin{definition}\label{spec-seq}
Let $\kappa$ be a cardinal. We say that an injective transfinite sequence
$(x_\alpha)_{\alpha<\kappa}$ of points of a $G$-act $X$ is {\em discriminate} if there is a (not
necessarily bijective) enumeration $G=\{g_\alpha\}_{\alpha<\kappa}$ of the monoid $G$ such that,
for all ordinals $\alpha<\kappa$ and $\beta,\gamma,\delta<\alpha$, the following conditions hold:
\begin{enumerate}
\item
if $g_\beta(x_0)\ne g_\gamma(x_0)$, then $g_\beta(x_\alpha)\ne g_\gamma(x_\alpha)$;
\item if $g_\beta(x_0)\ne g_\gamma(x_\delta)$, then $g_\beta(x_\alpha)\ne g_\gamma(x_\delta)$.
\end{enumerate}
\end{definition}

\begin{definition}\label{spec-fil}
A filter $\varphi$ on a $G$-act $X$ is said to be {\em discriminate} if
there is a cardinal $\kappa$ and a discriminate sequence $(x_\alpha)_{\alpha<\kappa}$ in
$X$ such that $\bigcap\varphi=\{x_0\}$ and $\{x_0\}\cup\{x_\beta:\beta>\alpha\}\in\varphi$ for all
ordinals $\alpha<\kappa$. In this case, the set $X_0=\{x_\alpha\}_{\alpha<\kappa}$ is called the
{\em discriminate support} of $\varphi$.
\end{definition}

For a discriminate filter $\varphi$ on a $G$-act $X$, the $G$-topology $\tau_\varphi$ has many
nice properties.

\begin{theorem}\label{t:spec}
For any discriminate filter $\varphi$ on a $G$-act $X$ with discriminate
support $X_0$ and intersection $\bigcap\varphi=\{x_0\}$, the $G$-topology $\tau_\varphi$ has the
following properties:
\begin{enumerate}
\item
the topological space $(X,\tau_\varphi)$ is hereditarily normal;
\item
for any $F\in\varphi$, the set $G(F)=\{g(x):g\in G,\;x\in F\}$ is  open and closed in
$(X,\tau_\varphi)$ and $X\setminus G(F)$ is discrete in $(X,\tau_\varphi)$\textup;
\item
$\{F\cap X_0:F\in\varphi\}=\{U\cap X_0:x_0\in U\in\tau_\varphi\}$\textup;
\item
$\psi(x_0,\tau_\varphi)=\psi(\varphi)$ and $\chi(x_0,\tau_\varphi)\ge\chi(\varphi)$.
\end{enumerate}
\end{theorem}

\begin{proof} By definition, the discriminate support $X_0$ of $\varphi$ admits an enumeration
$X_0=\{x_\alpha\}_{\alpha<\kappa}$ such that conditions (1) and (2) in
Definition~\ref{spec-seq} hold for some enumeration $G=\{g_\alpha\}_{\alpha<\kappa}$ of the monoid
$G$. If $\kappa$ is finite, then $\varphi$ consists of all subsets of $X$ containing $x_0$, so that
the topology $\tau_\varphi$ is discrete and, therefore, automatically possesses the required
properties. Thus, hereafter, we assume $\kappa$ to be infinite.

For every ordinal $\alpha<\kappa$,
consider the set $X_{>\alpha}=\{x_\beta:\alpha<\beta<\kappa\}$ and observe that $\{x_0\}\cup
X_{>\alpha}\in\varphi$ (by Definition~\ref{spec-fil}). Now, we shall prove the required properties
of the $G$-topology $\tau_\varphi$ as separate claims.

\begin{claim}
The topology $\tau_\varphi$ satisfies the separation axiom $T_1$.
\end{claim}

\begin{proof}
Given any point $x\in X$, we must show that $X\setminus\{x\}\in\tau_\varphi$.
Since the discriminate filter $\varphi$ contains all sets $\{x_0\}\cup X_{>\alpha}$ for
$\alpha\in \kappa$,
it suffices, given any map $g\in G$  with $g(x_0)\in X\setminus\{x\}$, to find
$\alpha<\kappa$ such that $g(X_{>\alpha})\subset X\setminus\{x\}$. If $x\notin G(X_0)$, then
$g(X_{>0})\subset G(X_0) \subset X\setminus\{x\}$, and we are done.

Suppose that $x\in G(X_0)$. Let us find ordinals $\gamma,\delta<\kappa$ such that
$x=g_\gamma(x_\delta)$ and, in addition, an ordinal $\beta<\kappa$ for which
$g_\beta=g$. Since $g_\beta(x_0)=g(x_0)\ne x=g_\gamma(x_\delta)$, it follows from condition (2)
in Definition~\ref{spec-seq} that $g(x_\alpha)=g_\beta(x_\alpha)\ne
g_\gamma(x_\delta)=x$ for all $\alpha>\max\{\beta,\gamma,\delta\}$. Consequently, for the ordinal
$\alpha=\max\{\beta,\gamma,\delta\}$, we have the required inclusion $g(X_{>\alpha})\subset
X\setminus\{x\}$.
\end{proof}

\begin{claim}
The topology $\tau_\varphi$ is hereditarily normal.
\end{claim}

\begin{proof} To prove the hereditary normality of the topology $\tau_\varphi$,
take any subspace $Y$ of the topological space $(X,\tau_\varphi)$ and choose two closed disjoint sets $A,B$ in $Y$. Let $A_0=\bar A\setminus \bar B$ and $B_0=\bar B\setminus \bar A$, where $\bar A$, $\bar B$ are the closures of the sets $A,B$ in $(X,\tau_\varphi)$. The equalities $A=\bar A\cap Y$ and $B=\bar B\cap Y$ imply that $A\subset A_0$ and $B\subset B_0$.

 Consider the sequences of sets $(A_n)_{n\in\w}$
and $(B_n)_{n\in\w}$ defined by recursion as
$$
A_{n+1}=A_n\cup\{g_\alpha(x_\gamma):\alpha<\gamma<\kappa,\;\;g_\alpha(x_0)\in A_n,\;
g_\alpha(x_\gamma)\notin B_n\cup\bar B\}
$$
and
$$
B_{n+1}=B_n\cup\{g_\beta(x_\delta):\beta<\delta<\kappa,\;\;g_\beta(x_0)\in
B_n,\;g_\beta(x_\delta)\notin A_n\cup\bar A\}.
$$

We claim that the sets  $A_\w=\bigcup_{n\in\w}A_n$ and $B_\w=\bigcup_{n\in\w}B_n$
are open disjoint neighborhoods of $A_0$ and $B_0$ in $(X,\tau_\varphi)$. First,
we check that these sets are disjoint.
Assuming the opposite, we can find numbers $n,m\in\w$ such that $A_{n+1}\cap B_{m+1}\ne\emptyset$
but $A_n\cap B_{m+1}=\emptyset=A_{n+1}\cap B_m$. Choose any point $c\in A_{n+1}\cap B_{m+1}$. By
the definition of the sets $A_{n+1}$ and $B_{m+1}$, the point $c$ is of the form
$g_\alpha(x_\gamma)=c=g_\beta(x_\delta)$ for some ordinals $\alpha<\gamma<\kappa$ and
$\beta<\delta<\kappa$ such that $g_\alpha(x_0)\in A_n$ and $g_\beta(x_0)\in B_m$. It follows from
$A_n\cap B_m=\emptyset$ that $g_\alpha(x_0)\ne g_\beta(x_0)$.
Condition (1) in Definition~\ref{spec-seq} guarantees that $\gamma\ne\delta$. Without
loss of generality, we can assume that $\delta>\gamma$.
Since $g_\beta(x_0)\ne g_\alpha(x_\gamma)$,
it follows from condition (2) in Definition~\ref{spec-seq} that
$g_\beta(x_\delta)\ne g_\alpha(x_\gamma)$; this contradiction shows that
$A_\w\cap B_\w=\emptyset$.

Now, let us show that the set $A_\w$ is open in $(X,\tau_\varphi)$. Given an ordinal
$\alpha<\kappa$ for which $g_\alpha(x_0)\in A_\w$, we must find a set $F\in \varphi$ such that
$g_\alpha(F)\subset A_\w$. Let $n\in\w$ be the smallest number for which
$g_\alpha(x_0)\in A_n$. We claim that the set $F=\{x_0\}\cup\{x\in X_{>\alpha}:g_\alpha(x)\notin
B_n\cup\bar B\}$ belongs to the filter $\varphi$. Assuming that $F\notin \varphi$, we conclude
that the set $X_{>\alpha}\setminus F$ belongs to the family $\varphi^+$, and the set
$E_k=\{x\in X_{>\alpha}:g_\alpha(x)\in B_k\cup\bar B\}$ with $k=n$  belongs to $\varphi^+$ as well.

Let $k\le n$ be the smallest number for which $E_k\in\varphi^+$.
We claim that $k>0$. Indeed, since $\bar B$ is a closed subset in $(X,\tau_\varphi)$, its complement
$X\setminus \bar B$ is an open neighborhood of the point $g_\alpha(x_0)\in A_n$. By the
definition of the topology $\tau_\varphi$, there is a set $F_0\in \varphi$ such that
$g_\alpha(F_0)\subset X\setminus \bar B$. Since $F_0$ intersects $E_k$ and is disjoint with
$E_0=\{x\in X_{>\alpha}:g_\alpha(x)\in \bar B\}$, we conclude that $k>0$.

Since $\varphi^+\not\ni E_{k-1}\subset E_k\in \varphi^+$, the set $E_k\setminus E_{k-1}$ is
non-empty and, hence, contains some point $x_\gamma$ with $\gamma>\alpha$. It follows that
$g_\alpha(x_\gamma)\in B_k\setminus B_{k-1}$; therefore,
$g_\alpha(x_\gamma)=g_\beta(x_\delta)$ for some ordinals $\beta<\delta<\kappa$ such that
$g_\beta(x_0)\in B_{k-1}$.

By condition (1) in Definition~\ref{spec-seq}, $\delta\ne\gamma$ (because
$g_\alpha(x_\gamma)=g_\beta(x_\delta)$ and  $g_\alpha(x_0)\ne g_\beta(x_0)$). If $\delta>\gamma$,
then the equality $g_\alpha(x_\gamma)=g_\beta(x_\delta)$ is excluded by condition
(2) in Definition~\ref{spec-seq}, since
$B_{k-1}\not\ni g_\alpha(x_\gamma)\ne g_\beta(x_0)\in
B_{k-1}$. If $\gamma>\delta$, then the equality $g_\alpha(x_\gamma)=g_\beta(x_\delta)$ is also
excluded by (2), because $A_n\ni g_\alpha(x_0)\ne
g_\beta(x_\delta)=g_\alpha(x_\gamma)\in B_k$. This contradiction shows that $F\in
\varphi$ and  $g_\alpha(F)\subset A_{n+1}\subset A_\w$, which means  the openness of
$A_\w$.

Similarly, the set $B_\w$ is open in $(X,\tau_\varphi)$.
Thus, $A_\w\cap Y$ and $B_\w\cap Y$ are disjoint open neighborhoods
of the sets $A=A_0\cap Y$ and $B=B_0\cap Y$ in $Y$,
which proves the normality of the topological $T_1$-space $Y$ and hereditary
normality of $(X,\tau_\varphi)$.
\end{proof}


The definition of the topology $\tau_\varphi$ implies that, for every $F\in\varphi$, the set
$G(F)=\{g(x):g\in G,\;x\in F\}$ is open and closed in $(X,\tau_\varphi)$ and $X\setminus
G(F)$ is discrete in $(X,\tau_\varphi)$.

\begin{claim}\label{cl:eq}
$\{F\cap X_0:F\in\varphi\}=\{U\cap X_0:x_0\in U\in\tau_\varphi\}$ and, therefore,
$\chi(\varphi)\le\chi(x_0,\tau_\varphi)$.
\end{claim}

\begin{proof} By the definition of the topology $\tau_\varphi$, we have
$\{U\in\tau_\varphi:x_0\in U\}\subset\varphi$; hence $\{U\cap X_0:x_0\in
U\in\tau_\varphi\}\subset\{F\cap X_0:F\in\varphi\}=\{F\in\varphi:F\subset X_0\}$.

To prove the reverse inclusion, fix any subset $F\in\varphi$ with $F\subset X_0$ and consider the
set $U=F\cup (X\setminus X_0)$. We claim that $U\in\tau_\varphi$.
To show this, given any ordinal $\alpha<\kappa$
for which $g_\alpha(x_0)\in U$, we must find a set $E\in\varphi$ such that $g_\alpha(E)\subset
U$. Take $\beta<\kappa$ for which $g_\beta=\id_X$ and consider the set
$$
E=\{x_0\}\cup\{x_\gamma\in F:\max\{\alpha,\beta\}<\gamma<\kappa\}\in\varphi.
$$
We have $g_\alpha(E)\subset U$. Indeed, assuming the converse, we can find an ordinal
$\gamma>\max\{\alpha,\beta\}$ such that $x_\gamma\in F$ and $g_\alpha(x_\gamma)\in X_0\setminus F$,
which implies $g_\alpha(x_\gamma)=x_\delta=\id_X(x_\delta)=g_\beta(x_\delta)$ for some ordinal
$\delta<\kappa$. Since $x_\gamma\in F$ and $x_\delta\notin F$, it follows that the ordinals
$\gamma$ and $\delta$ are distinct.

If $\gamma<\delta$, then the inequality $g_\beta(x_0)=x_0\ne x_\delta=g_\alpha(x_\gamma)$ and
condition (2) in Definition~\ref{spec-seq} imply $g_\beta(x_\delta)\ne
g_\alpha(x_\gamma)$, which is a contradiction.

If $\gamma>\delta$, then the inequality $g_\alpha(x_0)\ne x_\delta=g_\beta(x_\delta)$ and
condition (2) in Definition~\ref{spec-seq} imply $g_\alpha(x_\gamma)\ne
g_\beta(x_\delta)=x_\delta$, which again leads to  a contradiction.
\end{proof}

\begin{claim}\label{cl:l1}
The topology $\tau_\varphi$ has pseudocharacter $\psi(x_0,\tau_\varphi)=\psi(\varphi)$
at the point $x_0$.
\end{claim}

\begin{proof} The inequality $\psi(\varphi)\le\psi(x_0,\tau_\varphi)$ follows from
Claim~\ref{cl:eq}. To show that $\psi(x_0,\tau_\varphi)\le\psi(\varphi)$, fix a family
$\F\subset \varphi$ such that $|\F|=\psi(\varphi)$ and $\bigcap \F=\{x_0\}$. For every $F\in \F$,
we take the open neighborhood $U^F\in\tau_\varphi$ of $x_0$ being the union
$U^F=\bigcup_{n\in\w}U^F_n$ of the sequence of sets $(U^F_n)_{n\in\w}$
defined by the recursion as
$U^F_0=\{x_0\}$ and
$$
U^F_{n+1}=U^F_{n}\cup
\{g_\alpha(x_\beta):\alpha<\beta<\kappa,\;x_\beta\in F,\;\;g_\alpha(x_0)\in U^F_n\}\mbox{ \ for
every $n\in\w$}.
$$
The definition of the topology $\tau_\varphi$ implies that
$U^F=\bigcup_{n\in\w}U^F_n$ is an open neighborhood of $x_0$ in $X$.

Let us show that $\bigcap_{F\in \F}U^F=\{x_0\}$. Assume that, on the contrary,
this intersection contains a point $x$ distinct from $x_0$. For every $F\in \F$, find the smallest
number $n_F\in\w$ such that $x\in U^F_{n_F}$. Since $U^F_0=\{x_0\}\ne\{x\}$, we conclude
that $n_F>0$; hence $x\notin U^F_{n_F-1}$.
 By the definition of the set $U^F_{n_F}$, there are ordinals $\alpha_F<\beta_F<\kappa$ such that
$x_{\beta_F}\in F$ and  $x=g_{\alpha_F}(x_{\beta_F})\ne g_{\alpha_F}(x_0)\in U^F_{n_F-1}$.

We claim that there are two sets $F,E\in\varphi$ with $\beta_F\ne\beta_E$.
Fix any set $F\in \F$. Since $x_{\beta_F}\in F$ and $x_{\beta_F}\notin\{x_0\}= \bigcap \F$, there is a
set $E\in \F$ such that $x_{\beta_F}\notin E$. Then $\beta_F\ne \beta_E$. So, $F,E\in\varphi$ are two sets with $\beta_F\ne\beta_E$. We lose no generality assuming that $\beta_F<\beta_E$. Since $\beta_E>\max\{\alpha_E,\beta_F,\alpha_F\}$ and
$g_{\alpha_E}(x_0)\ne x=g_{\alpha_F}(x_{\beta_F})$, the condition (2) of Definition~\ref{spec-seq}
guarantees that $x=g_{\alpha_E}(x_{\beta_E})\ne g_{\alpha_F}(x_{\beta_F})=x$, which is the desired
contradiction that proves the equality $\bigcap_{F\in \F}U^F=\{x_0\}$ and the upper bound
$\psi(x_0,\tau_\varphi)\le\psi(\varphi)$.
\end{proof}
\end{proof}

\section{Zariski $G$-topology and the existence of discriminate filters}\label{s:existence}

In light of Theorem~\ref{t:spec}, it is of interest to describe $G$-acts
possessing discriminate sequences and discriminate filters.

\begin{proposition}\label{zar-spec} Suppose that $X$ is a $G$-act over a monoid $G$, $x_0\in X$,
and $\kappa$ is an infinite cardinal.
\begin{enumerate}
\item
If $|G|\le\kappa\le\psi(x_0,\zeta_G)$, then the $G$-act $X$ contains a discriminate sequence
$(x_\alpha)_{\alpha<\kappa}$.
\item
If the $G$-act $X$ contains a discriminate sequence
$(x_\alpha)_{\alpha<\kappa}$, then $|G|\le\kappa$ and $\cf(\kappa)\le\psi(x_0,\zeta_G)$.
\end{enumerate}
\end{proposition}

\begin{proof} 1. Suppose that $|G|\le\kappa\le\psi(x_0,\zeta_G)$ and
$G=\{g_\alpha:\alpha<\kappa\}$ is an enumeration of the monoid $G$ such that $g_0=1_G$.
Let us construct an injective transfinite sequence $(x_\alpha)_{\alpha<\kappa}$ of
points of the set $X$ such that, for any $\alpha<\kappa$ and $\beta,\gamma,\delta<\alpha$,
\begin{enumerate}
\item[(a)] if $g_\beta(x_0)\ne g_\gamma(x_0)$, then $g_\beta(x_\alpha)\ne
g_\gamma(x_\alpha)$;
\item[(b)] if $g_\beta(x_0)\ne g_\gamma(x_\delta)$, then $g_\beta(x_\alpha)\ne
g_\gamma(x_\delta)$.
\end{enumerate}
We construct this sequence by transfinite induction.

Assume that, for some ordinal $\alpha<\kappa$, the points $x_\beta$ with $\beta<\alpha$ have
already been constructed. For any ordinals $\beta,\gamma,\delta<\alpha$, consider the open
neighborhoods
$$
U_{\beta,\gamma}=\{x\in X:g_\beta(x_0)\ne g_\gamma(x_0)\;\Ra g_\beta(x)\ne
g_\gamma(x)\}
$$
and
$$
V_{\beta,\gamma,\delta}=\{x\in X:g_\beta(x_0)\ne g_\gamma(x_\delta)\;\Ra\;
g_\beta(x)\ne g_\gamma(x_\delta)\}$$
of $x_0$ in the Zariski $G$-topology $\zeta_G$. Since
$\psi(x_0,\zeta_G)\ge\kappa$, the intersection
$\bigcap_{\beta,\gamma,\delta<\alpha}U_{\beta,\gamma}\cap V_{\beta,\gamma,\delta}$ has cardinality
$\ge\kappa$ and hence  contains some point $x_\alpha\in X\setminus\{x_\beta:\beta<\alpha\}$.
Clearly,  this point $x_\alpha$ satisfies conditions (a) and (b).

\smallskip

2. Now, suppose that $X$ contains a discriminate sequence $X_0=\{x_\alpha\}_{\alpha<\kappa}$
for some infinite cardinal $\kappa$. Let $G=\{g_\alpha\}_{\alpha<\kappa}$ be an enumeration of the
monoid $G$ such that conditions (1) and (2) in Definition~\ref{spec-seq} are satisfied.
Then $|G|\le\kappa$.

Let us prove that $\psi(x_0,\zeta_G)\ge\cf(\kappa)$. Assume that, on the
contrary, there exists a family $\U\subset\tilde\zeta_G$ such that
$\bigcap\U=\{x_0\}$ and $|\U|<\cf(\kappa)$. For each set $U\in\U\subset\tilde\zeta_G$, choose
ordinals $\alpha_U,\beta_U<\kappa$ and a point $c_U\in X$ so that $U$ is equal to either
$\{x\in X:g_{\alpha_U}(x)\ne g_{\beta_U}(x)\}$ or $\{x\in X:g_{\alpha_U}(x)\ne c_{U}\}$. If
$c_U\in G(X_0)$, then we also choose ordinals $\gamma_U,\delta_U<\kappa$ such that
$c_U=g_{\gamma_U}(x_{\delta_U})$; otherwise, we set $\gamma_U=\delta_U=0$.
Since the set $A_U=\{\alpha_U,\beta_U,\gamma_U,\delta_U:U\in\U\}$ has cardinality $<\cf(\kappa)$,
there is an ordinal $\alpha<\kappa$ for which $\alpha>\sup A_U$. We claim that
$x_\alpha\in\bigcap\U$. Indeed, take any set $U\in\U$. If $U=\{x\in
X:g_{\alpha_U}(x)\ne g_{\beta_U}(x)\}$, then,
by condition~(1) in Definition~\ref{spec-seq}, we have $x_\alpha\in U$, because $x_0\in U$. If
$U=\{x\in X:g_{\alpha_U}(x)\ne c_U\}$ and $c_U\in G(X_0)$, then the relations
$c_U=g_{\gamma_U}(x_{\delta_U})$ and $x_0\in U$ and condition~(2) in
Definition~\ref{spec-seq} imply $x_\alpha\in U$. If $c_U\notin G(X_0)$, then
$g_{\alpha_U}(x_\alpha)\ne c_U$, and hence $x_\alpha\in U$. Therefore, $x_\alpha\in
\bigcap\U=\{x_0\}$, which is the desired contradiction.
\end{proof}

\section{$G$-topologizability of $G$-acts}\label{s:main}

In this section we apply the results of the preceding sections to prove the following
theorem, which is our main result.

\begin{theorem}\label{main0}
Let $X$ be a $G$-act over a monoid $G$  such that
$|G|\le\psi(x_0,\zeta_G)$ for some point $x_0\in X$. Then, for every infinite cardinal $\kappa$ satisfying $|G|\le\kappa\le \psi(x_0,\zeta_G)$ and every infinite cardinal $\lambda\le\cf(\kappa)$, the $G$-act $X$ admits $2^{2^\kappa}$
hereditarily normal $G$-topologies with pseudocharacter $\lambda$ at the point $x_0$.
\end{theorem}

\begin{proof} By Proposition~\ref{zar-spec}, the space $X$ contains a discriminate sequence
$X_0=\{x_\alpha\}_{\alpha<\kappa}$. Let $\varphi_0$ be the filter on $X$ generated by the sets
$\{x_0\}\cup\{x_\beta:\beta>\alpha\}$, where $\alpha<\kappa$, and let
${\uparrow}\varphi_0$ denote the set of all filters $\varphi$ on $X$ which contain the filter
$\varphi_0$ and satisfy the condition $\cap\varphi=\{x_0\}$.

\begin{claim}
For any infinite cardinal $\lambda\le\cf(\kappa)$, the set
$\F_\lambda=\{\varphi\in{\uparrow}\varphi_0:\psi(\varphi)=\lambda\}$ has cardinality $|\mathcal
F_\lambda|=2^{2^\kappa}$.
\end{claim}

\begin{proof} First, observe that the family of all filters on the set $X_0$ has cardinality $\le
2^{2^\kappa}$. Thus, $|\F_\lambda|\le 2^{2^\kappa}$. To prove the reverse inequality, we consider
two cases. \smallskip

\emph{Case 1}: $\lambda=\cf(\kappa)$. Let us represent the set $X_0$ as the disjoint union
$X_0=X_0'\cup X_0''$ of two sets of cardinality $|X_0'|=|X_0''|=\kappa$ such that $x_0\in X_0'$. On
the set $X_0'$ consider the filter $\varphi_0|X_0'=\{F\cap X_0':F\in\varphi_0\}$. Posp\'i\v sil's
Theorem~\cite{Pos} (see also \cite{Jech}) implies that the family $\U_0$ of all ultrafilters on
$X_0''$ which contain the filter  $\varphi_0|X_0''=\{F\cap X_0'':F\in\varphi_0\}$ has cardinality
$2^{2^\kappa}$. For any ultrafilter $u\in\U_0$, consider the filter $\varphi_u=\{A\subset X:A\cap
X_0''\in u,\;A\cap X_0'\in\varphi_0|X_0'\}$ and note that
$\psi(\varphi_u)=\psi(\varphi_0|X_0')=\cf(\kappa)$. For distinct ultrafilters $u,v\in\U_0$,
the filters $\varphi_u$ and $\varphi_v$ are distinct; therefore,
$|\F_{\cf(\kappa)}|\ge2^{2^\kappa}$.
\smallskip

\emph{Case 2}: $\lambda<\cf(\kappa)$. In this case the ordinal $\kappa$ can be identified with the
product $\kappa\times\lambda$ endowed with the lexicographic order,
in which $(\alpha,\beta)<(\alpha,\beta')$ if and only if either $\alpha<\alpha'$ or
$\alpha=\alpha'$ and $\beta<\beta'$. Let $\xi:\kappa\times\lambda\to\kappa$ be the order
isomorphism. Consider the filter $\varphi_\lambda$ of cofinite subsets
on the cardinal $\lambda$. This filter has pseudocharacter $\psi(\varphi_\lambda)=\lambda$.
As shown in Case~1, the family
$\F_{\cf(\kappa)}$ has cardinality $2^{2^\kappa}$. Given any filter $u\in\F_{\cf(\kappa)}$,
consider the filter $\varphi_u$ on $X$ generated by the sets
$$
\Phi_{U,L}=\{x_0\}\cup\{x_{\xi(\alpha,\beta)}:x_\alpha\in U,\;\beta\in L\}\mbox{ where }U\in
u,\;L\in\varphi_\lambda.
$$
It can be shown that $\psi(\varphi_u)=\psi(\varphi_\lambda)=\lambda$ and
the filters $\varphi_u$ and $\varphi_v$ are
distinct for any distinct filters $u,v\in \F_{\cf(\kappa)}$. Consequently,
$|\F_\lambda|\ge|\F_{\cf(\kappa)}|\ge 2^{2^\kappa}$.
\end{proof}

According to Theorem~\ref{t:spec}, for any filter $\varphi\in\F_\lambda$, the $G$-topology
$\tau_\varphi$ on $X$ is hereditarily normal and has pseudocharacter
$\psi(x_0,\tau_\varphi)=\psi(\varphi)=\lambda$ at $x_0$. Theorem~\ref{t:spec}(3) implies that
distinct filters $u,v\in\F_\lambda$  determine distinct
topologies $\tau_u$ and $\tau_v$.
Therefore, $X$ admits at least $|\F_\lambda|=2^{2^\kappa}$ hereditarily normal $G$-topologies with
pseudocharacter $\lambda$ at $x_0$.
\end{proof}

\begin{remark}
Example~\ref{ex2.7} shows that the condition
$|G|\le\kappa\le\psi(x_0,\zeta_G)$ in Theorem~\ref{main0} is sufficient but not necessary
for normal $G$-topologizability: the
group $G=H^\kappa$ is normally $G_s$-topologizable but
$\psi(x_0,\zeta_{G_s})=\kappa<2^\kappa=|G_s|$ for any point $x_0$.
On the other hand, this condition is necessary for the existence of a Hausdorff
$G$-topology with pseudocharacter $|G|$ (indeed, it follows directly from
Proposition~\ref{Zarisweak} that $\psi(x_0,\zeta_{G_s})\ge \psi(x_0,\tau)$ for any
Hausdorff $G$-topology $\tau$ on $X$).
Thus, for acts over monoids of
regular cardinality,
Theorem~\ref{main0} has the following corollary:
\emph{Let $X$ be a $G$-act over a monoid $G$ of regular cardinality $\kappa=|G|$,
and let $x_0\in X$. Then
$\psi(x_0,\zeta_G)=\kappa$ if and only if $X$ admits a Hausdorff
$G$-topology $\tau$ with $\psi(x_0,\tau)=\kappa$.}
\end{remark}

\begin{remark}
For each of the $S$-acts $G$ over $S\in\{G_l, G_r, G_s,
G_q, G_{[s]}, G_{[q]}, G^+_p, G_{p}\} \cup \{G_m:m\in\IZ\}$ defined in Section~\ref{s:zar}
and any infinite cardinals $\lambda$ and $\kappa$, $\lambda\le\cf(\kappa)$,
the condition $|G|\le\kappa\le\psi(x_0,\zeta_S)$ for some point $x_0\in X$
implies the existence of $2^{2^\kappa}$ hereditarily normal $S$-topologies with
pseudocharacter $\lambda$ at the point $x_0$. In particular, any group $G$ admits
$2^{2^{|G|}}$ hereditarily normal left-invariant  and $2^{2^{|G|}}$ hereditarily normal right-invariant
topologies.
\end{remark}

In the case of countable monoids $G$, Theorem~\ref{main0} implies the following
characterization of $G$-topologizability, which solves Problem~\ref{pr2} in this case.

\begin{theorem}\label{main}
For a countable monoid $G$ and a $G$-act $X$,
the following conditions are equivalent:
\begin{enumerate}
\item $X$ admits a non-discrete Hausdorff $G$-topology;
\item $X$ admits a non-discrete metrizable $G$-topology;
\item the Zariski $G$-topology $\zeta_G$ on $X$ is non-discrete;
\item $X$ admits $2^{\mathfrak c}$ non-discrete hereditarily normal $G$-topologies.
\end{enumerate}
\end{theorem}

\begin{proof} The equivalence $(1)\Leftrightarrow(3)\Leftrightarrow(4)$
follows from Theorems~\ref{t:spec} and \ref{main0}. The implication $(2)\Rightarrow (1)$
is trivial. It remains to prove that $(3)\Rightarrow(2)$.
If the Zariski $G$-topology $\zeta_G$ on $X$ is non-discrete,
then some point $x_0\in X$ has infinite pseudocharacter $\psi(x_0,\zeta_G)$.
By Proposition~\ref{zar-spec}, the $G$-act contains a discriminate sequence $(x_n)_{n<\w}$.
Let $\varphi$ be the discriminate filter on $X$ generated by the base consisting
of the sets $\{x_0\}\cup\{x_m:m\ge n\}$, $n\in\w$, and let
$\tau_\varphi=\{U\subset X:\forall g\in G\;g^{-1}(U)\in\varphi\}$ be the $G$-topology
on $X$ generated by the filter $\varphi$. By Theorem~\ref{t:spec}(2),
the orbit $G(X_0)$ of the set $X_0=\{x_n\}_{n<\w}$ is open in $(X,\tau_\varphi)$,
and its complement $X\setminus G(X_0)$ is an open discrete subspace of $X$.

By Theorem~\ref{t:spec}(1,4), the topology $\tau_\varphi$
is non-discrete and hereditary normal, which implies that the open subspace $G(X_0)$
is Tychonoff. Therefore, given any distinct points $x,y\in G(X_0)$, we can find
a continuous function $f_{x,y}:G(X_0)\to[0,1]$ such that $f_{x,y}(x)=0$ and $f_{x,y}(y)=1$.
Since the monoid $G$ is countable, so is the function family
$$
\F=\{f_{x,y}\circ g\colon g\in G,\;\;x,y\in G(X_0),\;x\ne y\}.
$$
Choose any enumeration $\F=\{f_n\}_{n\in\w}$ of the family $\F$ and consider
the continuous metric $d$ on $G(X_0)$ defined by
$$
d(x,y)=\max_{n\in\w}\frac1{2^n}{\big|f_n(x)-f_n(y)\big|}.
$$
We extend the metric $d$ to the whole space $X$ by setting $d(x,x)=0$
for any $x\in X$ and $d(x,y)=1$ for any pair $(x,y)\in X^2\setminus G(X_0)^2$ with $x\ne y$.
Since $X\setminus G(X_0)$ is an open-and-closed discrete subspace of $X$,
the  metric $d$ thus extended remains continuous. It can be shown that
each function $g\in G\subset X^X$
is continuous with respect to the metric $d$; hence the topology $\tau_d$ on $X$
generated by the metric $d$ is a $G$-topology, which is contained in the $G$-topology
$\tau_\varphi$ and therefore not discrete. Thus, $\tau_d$ is the required non-discrete
metrizable $G$-topology on $X$.
\end{proof}

Note that Theorem~\ref{main} applies also to finite monoids $G$. Indeed, it is easy
to reduce the finite case to the infinitely countable one by considering
the direct product $G\times S$ of a given finite monoid $G$ and any infinite
countable monoid $S$ acting trivially on $X$. Clearly, for the product action
of $G\times S$, we have $\zeta_G = \zeta_{G\times S}$, and any $G$-topology
on $X$ is a $(G\times S)$-topology.

We do not know whether this theorem remains valid for arbitrary $G$-acts.

\begin{problem}\label{genpro}
Let $G$ be an uncountable monoid (group), and let $X$ be a  $G$-act
for which the Zariski $G$-topology $\zeta_G$ is non-discrete. Is it true that $X$ is
$G$-topologizable?
\end{problem}

In solving this problem, results of \cite{HPS} may be useful.

The general Problem~\ref{genpro} is also interesting in special cases, for example, where $X$ is a
(semi)group, the action is assumed to satisfy additional conditions (e.g., be
transitive, faithful, or free), or the orbit space has certain properties (e.g., is countable).

\section{Acknowledgement}

The authors would like to thank the anonymous referee for careful reading the manuscript and useful comments, which helped the authors to prove new results and improve the
exposition.


\begin{thebibliography}{99}


\bibitem{Arh}
A.V. Arhangel'skii, {\em
Cardinal invariants of topological groups. Embeddings and condensations},
Dokl. Akad. Nauk SSSR, 247 (1979), 779--782;
English translation: Soviet Math. Dokl. 20 (1979), 783--787.



\bibitem{BG}
T.~Banakh, I.~Guran, {\em Perfectly supportable semigroups are $\sigma$-discrete in each Hausdorff shift-invariant topology,}
Topol. Algebra Appl. 1 (2013), 1--8.

\bibitem{BGP} T.~Banakh, I.~Guran, I.~Protasov,
{\em Algebraically determined topologies on
permutation groups,}
Topology Appl.
159 (9) (2012), 2258--2268.

\bibitem{BPS}
T.~Banakh, I.~Protasov,
{\em Topologizations of $G$-spaces,}
Ukrain. Math. Bull.
9 (3) (2012), 308--317.

\bibitem{Br}
R.~Bryant,
{\em The verbal topology of a group,}
J. Algebra
48 (2) (1977), 340--346.

\bibitem{DikShakh}
D.~Dikranjan, D.~Shakhmatov,
{\em Selected topics from the structure theory of topological groups,} in:
E.M.~Pearl (Ed.),
Open problems in topology II,
Elsevier Science B.V., Amsterdam, 2007,
pp.~389--406.

\bibitem{DS1}
D.~Dikranjan, D.~Shakhmatov,
{\em Reflection principle characterizing groups in which
unconditionally closed sets are algebraic,}
J. Group Theory 11 (3) (2008), 421--442.

\bibitem{DS2}
D.~Dikranjan, D.~Shakhmatov,
{\em The Markov--Zariski topology of an abelian group,}
J. Algebra
324 (6) (2010), 1125--1158.

\bibitem{DP} D.~Dikranjan, I.~Protasov,
{\em Counting maximal topologies on countable groups and rings},
Topology Appl., {\bf 156} (2008), 322-325.

\bibitem{DT1}
D.~Dikranjan, D.~Toller,
{\em Markov's problems through the looking glass of Zariski
and Markov topologies,}
in: Ischia Group Theory 2010,
Proceedings of the Conference, World Scientific Publ.,
Singapore, 2012,
pp.~87--130.

\bibitem{DT}
D.~Dikranjan, D.~Toller,
{\em The Markov and Zariski topologies of some linear groups,}
Topology Appl. 159 (13) (2012), 2951--2972.

\bibitem{Hesse} G. Hesse,
{\em Zur Topologisierbarkeit von Gruppen, Dissertation,} Univ. Hannover,
Hannover, 1979.

\bibitem{HPS}
N.~Hindman, I.~Protasov, D.~Strauss,
{\em Topologies on $S$ determined by idempotents in $\beta S$,}
Topol. Proc. 23 Summer (1998), 155--190.

\bibitem{Jech}
T.~Jech,
{\em Set Theory,}
Springer-Verlag, Berlin, 2003.

\bibitem{KKM}
M.~Kilp, U.~Knauer, A.V.~Mikhalev,
{\em Monoids, Acts and Categories. With Applications to Wreath Products and Graphs,}
Expositions in Math., vol.~29, Walter de Gruyter, Berlin, 2000.

\bibitem{KOO}
A.A.~Klyachko, A.Yu.~Olshanskii, D.V.~Osin,
{\em On topologizable and non-topologizable groups,} Topology Appl. 160 (16) (2013) 2104--2120.

\bibitem{KT}
A.A.~Klyachko, A.V.~Trofimov,
{\em The number of non-solutions of an equation in a group,}
J. Group Theory 8 (6) (2005), 747--754.


\bibitem{Markov1944}
A.A. Markov,
{\em On unconditionally closed sets,}
Dokl. Akad. Nauk SSSR
44 (5) (1944), 196--197;
English translation:
A. Markoff, On unconditionally closed sets,
C.R. (Doklady) Acad. Sci. URSS (N.S.) 44 (1944), 180--181.



\bibitem{Markov1945}
A.A.~Markov,
{\em On free topological groups}, Izv. Akad. Nauk SSSR, Ser. Mat.
9 (1) (1945), 3--64;
 English translation:
Three papers on topological groups: I. On the existence of periodic  connected topological
groups. II. On free topological groups. III. On  unconditionally closed sets,
Amer. Math. Soc. Transl. 30 (1950).


\bibitem{Mar}
A.A.~Markov,
{\em On unconditionally closed sets,}
Mat. Sb. 18 (1) (1946), 3--28;
English translation:
Three papers on topological groups: I. On the existence of periodic  connected topological
groups. II. On free topological groups. III. On  unconditionally closed sets,
Amer. Math. Soc. Transl. 30 (1950).


\bibitem{Ol80}
A.Yu.~Ol'shanskii,
{\em A remark on a countable nontopologized group,}
Vestnik Moscow Univ. Ser. I. Mat. Mekh, no.~3 (1980), 103.

\bibitem{Ol89} A.Yu.~Ol'shanskii, {\em Geometry of defning relations in groups}, Nauka, Moscow, 1989;
English translation in Math. And Its Applications (Soviet series), 70, Kluwer Acad. Publishers, Dordrecht, 1991.

\bibitem{Podewski}
K.-P.~Podewski,
{\em Topologisierung algebraischer Strukturen,}
Rev. Roumaine Math. Pures Appl.
22 (9)  (1977), 1283--1290.

\bibitem{Pos}
B.~Posp\'\i \v sil,
{\em Remark on bicompact spaces,}
Ann. of Math. (2) 38 (4) (1937), 845--846.






\bibitem{Sipa}
O.~Sipacheva,
{\em Unconditionally $\tau$-closed and $\tau$-algebraic sets in
groups,}
Topol. Appl. 155 (4) (2008), 335--341.

\bibitem{Zel}
E.~Zelenyuk,
{\em On topologies on groups with continuous shifts and inversion,}
Visn. Kyiv. Univ. Ser. Fiz.-Mat. Nauki, no. 2 (2000), 252--256.

\bibitem{Zel07}
Y.~Zelenyuk,
{\em On topologizing groups,} J. Group Theory 10 (2) (2007),
235--244.

\end{thebibliography}
\end{document}